\documentclass[11pt,final]{amsart}
\usepackage{amsmath,amsthm,amssymb}

\usepackage[notref]{showkeys}
\usepackage[dvips]{graphicx}

\usepackage{color,epsfig}
\usepackage{xcolor}

\usepackage{bbm}
\usepackage[all]{xy}
\usepackage{amscd}
\usepackage{dsfont}
\usepackage{verbatim}

\usepackage[english]{babel}
\usepackage[T1]{fontenc}

\newcommand{\convprob}{\stackrel{\PP}{\rightarrow}}

\newtheorem{thm}{Theorem}[section]
\newtheorem{lem}[thm]{Lemma}

\newtheorem{prop}[thm]{Proposition}

\theoremstyle{remark}

\theoremstyle{definition}

\newtheorem{df}[thm]{Definition}

\newtheorem{assumption}{Assumption}

\numberwithin{equation}{section}

\def\R{{\mathbb{R}}}

\newcommand{\ind}{\mathbbm{1}}
\newcommand{\EE}{\mathbb{E}}
\newcommand{\PP}{\mathbb{P}}

\newcommand{\NN}{\mathbb{N}}
\newcommand{\N}{\mathbb{N}}
\newcommand{\M}{\mathbb{M}}
\newcommand{\T}{\mathbb{T}}

\newcommand{\RR}{\mathbb{R}}

\newcommand{\Acal}{\mathcal{A}}

\newcommand{\ud}{\mathrm{d}}

\usepackage{soul}

\newcommand{\be}{\begin{equation}}
\newcommand{\ee}{\end{equation}}

\def\ba#1\ea{\begin{align*}#1\end{align*}}
\def\ban#1\ean{\begin{align}#1\end{align}}

\renewcommand{\ud}{{\rm d}}
\newcommand{\bsymb}{\boldsymbol}

\newcommand{\Pp}{ \mathbf{V} }% The point process associated with the BRW

\newcommand{\Tp}{ \mathbf{T} }% The point process associated with the bigest jumps in BRW
\newcommand{\Npa}{ \mathbf{N} }% The point process associated with the jumps in BRW
\newcommand{\Npb}{ \mathbf{N}^* }% The decorated varsion of point process associated with jumps in  BRW

\newcommand{\const}{ {\rm c} } % generic constant
\newcommand{\U}{\mathbb{U}} % set of Ulam-Harris labels
\newcommand{\bR}{ \overline{\mathbb{R}} }% The underlying space for measures

\newcommand{\shift}{\mathfrak{T}}
\newcommand{\scale}{\mathfrak{S}}
\newcommand{\tip}{\mathfrak{V}}
\newcommand{\sL}{\mathfrak{L}}

\begin{document}
	\title{Branching random walk and log-slowly varying tails}
	\author[A. Bhattacharya, P. Dyszewski, N. Gantert and~Z.~Palmowski]{Ayan Bhattacharya, Piotr Dyszewski, Nina Gantert and~Zbigniew~Palmowski}
	\date{\today}

%	\address{Instytut Matematyczny, Uniwersytet Wroclawski, Plac Grunwaldzki 2/4, 50-384 Wroclaw, Poland}
%	\email{\ }

	\thanks{ The research of PD was partially supported by the National Science Centre, Poland (Sonata, grant number 2020/39/D/ST1/00258).
 The research of ZP was partially supported by the National Science Centre, Poland (Opus, grant number 2021/41/B/HS4/00599). }

	\begin{abstract}
		We study a branching random walk with independent and identically distributed, heavy tailed displacements. The offspring law is supercritical and satisfies the Kesten-Stigum condition.
We treat the case when the law of the displacements does not lie in the max-domain of attraction of an extreme value distribution. Hence,
		the~classical extreme value theory,
		which is often deployed in this kind of models, breaks down. We show that if the tails of the displacements are such that the absolute value of the logarithm of the tail is a
		slowly varying function, one can still
		effectively analyse the extremes of the process. More precisely, after a non-linear transformation the extremes of the
branching random walk process converge to a cluster
		Cox process.
	\end{abstract}

	\maketitle

\section{Introduction}\label{sec:Introduction}

	Consider a system of particles which evolves on the real line $\R$ in the following way. In the very beginning the system consists of one
	particle placed at the origin $0$ of $\R$. At each time epoch $n \in \NN = \{ 1, 2, \ldots\}$ the particles present in the system
	split independently
	with a prescribed reproduction law, next each new particle is shifted independently from its place of birth according to a~fixed displacement distribution.
	Let $\Pp_n\subseteq \R$ denote the point process of the positions of particles during the $n$th epoch of time.
	The sequence $\{ \Pp_n\}_{n \in \NN }$ is a branching random walk (BRW). \medskip

	Since the early study done in the seventies by Biggins~\cite{76:biggins:first}, Kingman~\cite{75:kingman:first} and
	Hammersley~\cite{J.M.Hammersley1974} BRW became an active field of probability theory. The~process is a discrete time analogue of
	branching Brownian motion and thus bears connections to the Fisher-Kolmogorov-Petrovsky-Piskunov~\cite{kolmogorov1937etude, fisher1937wave} equation
	(see~\cite{mckean1975application, bovier2017gaussian} for more details).
	Apart from a natural interpretation
	of modelling fitness in a population, branching random walks are closely connected to fragmentation
	processes~\cite{kyprianou2017, dadoun:2017}, Mandelbrot cascade measures~\cite{Liu2000a,barral:2014}  and the so-called
	smoothing transform~\cite{alsmeyer:2012}. The aforementioned FKPP equation arises in the study of the extremal positions of the particles.
	For this reason the position of, say the rightmost particle $M_n = \max \Pp_n$ was one of the central objects of the probabilistic analysis of BRW.
	As one may expect the asymptotic behaviour of $M_n$ is closely related to the asymptotic properties of the displacement distribution.
	In the classical light-tailed case, when the displacement distribution has some finite exponential moments,
	$M_n$ moves at a  linear speed and after an additional logarithmic correction exhibits a weak limit which is a random shift of the
	Gumbel law~\cite{13:aidekon:convergence}. In the heavy tailed case, when the displacements have no finite exponential moments but have some
	finite power moments, one relies on the extreme value theory to show that $M_n$ grows faster
	than linear but at most exponentially with a weak limit which is a random shift of the
	Gumbel~\cite{dyszewski2022extremal} or Frechet~\cite{Durrett1983}
	law depending on the tails of the displacements. We show that one can still rely on the extreme value theory if the displacements have
	log-slowly varying tails, i.e. if the logarithm of the tail function is slowly varying.
	However the classical extreme value theory breaks down if the displacements are
	too heavy, i.e. don't have any power moments. It turns out that in this case there is no linear transformation under which one could obtain a
	limit theorem for $M_n$. In the present article we give a non-linear scale under which $M_n$ moves linearly
	with a weak limit which is a random shift of the Gumbel law. Moreover, we analyse the limiting behaviour of the extremal process for the
	BRW which describes the positions of the particle in the vicinity of the rightmost one. \medskip

	Assuming that the underlying genealogical structure forms a supercritical Galton-Watson process we will investigate the
	case of independent identically distributed (iid) displacements with the tails written in the form
	\begin{equation*}
		\PP[X>t] = a(t)e^{-L(t)},
	\end{equation*}
	for $t \in \RR$, where $a(\cdot)$ is a bounded, measurable function such that $a(t) \to a>0$ as $t \to \infty$.
	The asymptotic behaviours of $\Pp_n$ and $M_n$
	depend on the asymptotic behaviour of $L$. In order to motivate our setting we will briefly discuss some known results concerning
	limit theorems for $M_n$.\medskip

	The classical case, when the displacements have some exponential moments, corresponds to $L$ increasing at least linearly fast
	\begin{equation}\label{eq:1:thin}
		\liminf_{t \to \infty}L(t)/t >0
	\end{equation}
	and was studied in~\cite{75:kingman:first, 76:biggins:first, J.M.Hammersley1974, Hu2009, 13:aidekon:convergence}.
	Under some mild technical assumptions, one can find real, positive constants $\const_1$, $\const_2$ such that
	\begin{equation*}
		M_n - \const_1 n - \const_2 \log n
	\end{equation*}
	converges in distribution to a law with the cumulative distribution function (cdf) that can be expressed as
	\begin{equation*}
		x \mapsto \EE \left[\exp\{-D_\infty e^{-x} \} \right],
	\end{equation*}
  	where $D_\infty$ is the limit of so-called derivative martingale associated with the BRW. We refer to~\cite{Shi2015} for a self-contained treatment.
	Note that in this case the local behaviour of $L$ affects the limiting distribution.\medskip

	Moving to displacements with heavier tails, we will now discuss the log-regularly varying case.
	If $L$ is regularly varying with index $r \in (0,1)$, that is
	\begin{equation}\label{eq:1:rv}
		\lim_{t \to \infty} L(\const t)/L(t) = \const^r,
	\end{equation}
  	for any $\const>0$, the position of the rightmost particle grows polynomially in $n$. More precisely, under some regularity assumptions on $L$,
	for some slowly varying functions $\ell_1$, $\ell_2$,
	\begin{equation*}
		(M_n - n^{1/r} \ell_1(n))/ (n^{1/r-1}\ell_2(n))
	\end{equation*}
	converges in law to a distribution with cdf
	\begin{equation}\label{eq:1:mixGumbel}
		x \mapsto \EE \left[ \exp\{ - W e^{-x} \} \right],
	\end{equation}
  	where $W$ is the limiting value of the martingale associated with the underlying
	Galton-Watson process (see Section~\ref{sec:2:main} for details).
	This case was first treated in~\cite{Gantert2000} wich was followed by more recent work~\cite{dyszewski2023maximum, dyszewski2022extremal}.
	In the present paper we study a boundary case when $r=0$. This means that $L$ is slowly varying, that is,
	\begin{equation*}
		\lim_{t \to \infty} L(\const t)/L(t) = 1,
	\end{equation*}
	for any positive constant $c$.
	As one of our results shows, if the function $L$ is suplogarithmic, that is $\log t /L(t) \to 0$,
	then there are subexponentially growing sequences $\{ b_n\}_{n \in \NN}$ and
	$\{ a_n\}_{n \in \NN}$ such that
	\begin{equation*}
		(M_n - b_n)/a_n
	\end{equation*}	
	converges in law to a distribution with cdf of the form~\eqref{eq:1:mixGumbel}. Here and in what follows a sequence of real numbers
	$\{x_n\}_{n \in \NN}$ is said to be growing subexponentially if $x_n e^{-\varepsilon n} \to 0$ for any $\varepsilon >0$.
	One notable example is the lognormal distribution for which $L(t) = ( (\log t)^2 + 2 \log \log t)/2$.\medskip

	The case of logarithmically growing $L$ in essence boils down to steps with regularly varying tails. If one assumes for simplicity that
	$L(t) = \alpha \log t $ for some $\alpha>0$, then
	\begin{equation*}
		m^{-n/\alpha}M_n
	\end{equation*}
	converges in law to a random shift of a Fréchet distribution~\cite{Durrett1983}, that is, to a law with a cdf
  	\begin{equation*}%\label{eq:1:mixGumbel}
    		x \mapsto \EE \left[ \exp\{ - W x^{-\alpha} \} \right],
  	\end{equation*}
	where as previously $W$ is the martingale limit associated with the underlying Galton-Watson process. \medskip

	Finally, in the sublogarithmic case, that is when $L(t)/\log t \to 0$, the classical extreme value theory breaks down. Note that the last condition
	implies $\EE [|X|^\varepsilon] =\infty$ for any positive $\varepsilon>0$. However, one can still analyse $M_n$
	under a non-linear scale. As our main results show,
	\begin{equation*}
		L(M_n) - n \log m
	\end{equation*}
	converges in law to a distribution with a cdf given via~\eqref{eq:1:mixGumbel}.\medskip

	As it is usually the case for heavy tails, we were able to analyse not only the rightmost position $M_n$, but the joint behaviour of the rightmost
	extremes of the BRW $\Pp_n$. For this reason we formulate our main results in terms of point process convergence in a right point
	compacification of the real line.\medskip

	The article is organized as follows. In Section~\ref{sec:2:main} we give a precise set-up and the statements of our main results, Theorem \ref{thm:2:1} and
Theorem \ref{thm:L:sublog}.
	Section~\ref{sec:3:RWestimates} provides estimates for random walks used in the proofs presented in Section~\ref{sec:4:proofs}.
	Throughout the article we write $f(t) = o(g(t))$ for functions $f,g \colon \RR \to \RR$ if $f(t)/g(t) \to 0$ as $t \to \infty$. We also
	write $\const$ for a generic constant whose value is of no importance for us. Note that the actual value of $\const$ may
	change from line to line.

\section{Main results}\label{sec:2:main}

	We will now provide a detailed introduction of our model followed by the statement of our main results. One of our standing assumptions
	is that the genealogical structure is independent from the displacements and hence  we will introduce
	both components separately. We begin with the underlying Galton-Watson process, which we will introduce by explicitly describing the
	corresponding random tree. \medskip

	Denote by $\U = \{\emptyset\} \cup \bigcup_{k=1}^\infty \N^k$, with $\N = \{ 1, 2, \ldots\}$ the set of all possible labels.
	For $x,y \in \U$ write $x \leq y$ if $y_{|k} =x$ for some $k \geq 1$ and $|x|=n$ if $x \in \N^n$. Additionally, for $x \in \U$
	write $x_0 = \emptyset$ and $x_k$ for the projection of $x$ onto the $k$th coordinate for $k \leq |x|$. Consider a family
	$\{ N(x) \: : \: x \in \U \}$ of iid random variables taking values in $\N\cup\{0\}$.
	We will interpret $N(x)$ as the number of children of the particle labelled with $x \in \U$.  Define the aforementioned random tree via
	\begin{equation*}
		\T = \left\{ x \in \U \: : \: x_k \leq N(x_{k-1}) \mbox{ for all } 1\leq k \leq |x| \right\}.
	\end{equation*}
	Then $\T$ can be given a tree structure by placing an edge between any two vertices $x , y \in \T$ such that $ y = (x, k)$, $k \leq N(x)$.
	Notice that $\T$ is the genealogical tree corresponding to the Galton-Watson process
	\begin{equation*}
		Z_n = \# \T_n , \qquad \T_n=  \{ x \in \T \: : \: |x|=n \}.
	\end{equation*}
	Notice also that $Z_1 = N(\emptyset)$.
	In what follows we will impose conditions on $Z_1$ which will allow to control the growth of the process
	$\{ Z_n \}_{n \in \NN}$.

	\begin{assumption}\label{as:galton}
		We assume that the Galton-Watson process $\{Z_n \}_{n \in \NN}$ is supercritical, i.e. $\EE[Z_1]=m\in (1, \infty)$.
		Furthermore, we assume that
		$\EE[Z_1 \log^+(Z_1)] <\infty$.
	\end{assumption}
	The first imposed condition ensures that the process $\{Z_n\}_{n \in \NN}$ survives with positive probability, more precisely, we have then
	\begin{equation*}
		\PP [ Z_n \geq 1, \forall n \geq 1]>0.
	\end{equation*}
	Since our considerations become void when the population dies out, one can restrict the considerations to the set of survival and work
	under the conditional probability
	\begin{equation*}
		\PP^*[ \: \cdot \: ] = \PP [ \:  \cdot \: | \:  Z_n \geq 1, \forall n \geq 1 ].
	\end{equation*}
	One can then recast our main results under $\PP^*$.
	The second condition in Assumption~\ref{as:galton} secures a strict exponential growth of $Z_n$. To describe it precisely first note the the
	sequence $\{m^{-n}Z_n\}_{n \in \NN}$ forms a non-negative $\PP$-martingale and  therefore converges to, say,
	\begin{equation}\label{eq:2:GaltonLimit}
		\lim_{n \to \infty} m^{-n}Z_n = W \quad \PP-{\rm a.s.}
	\end{equation}
	It turns out that under the Assumption~\ref{as:galton}, $\PP^*[W>0]=1$ and the above convergence takes place in $L^1(\PP)$,
	see~\cite[Theorem I.10.1]{athreya:1972:branching}. The distribution of the limiting random variable enjoys a self-similarity property which
	will translate to an
	analogous property of the limiting measure of $\{\Pp_n\}_{n \in \NN}$. If for $x \in \T_1$ we write
	$Z_n(x) =\#\{ y \in \T \: : \: y\geq x, \: |y|=n+|x| \}$, then conditioned on $x \in \T_1$ the sequence $\{ Z_n(x)\}_{n \in \NN}$ is distributed
	as $\{Z_{n}\}_{n \in \NN}$ and thus $\lim_{n \to \infty} m^{-n} Z_n(x) = W_x$ $\PP-{\rm a.s.}$
	Using $Z_n = \sum_{|x|=1} Z_{n-1}(x)$ we infer that
	\begin{equation}\label{eq:2:ST}
		W = \frac 1m \sum_{|x|=1} W_x.
	\end{equation}
	The key feature of this formula is that $W_x$'s are iid copies of $W$ independent from the length of the sum $N = Z_1 = \# \T_1$.
	The random variable  $W$ will play the role of a shift parameter in the directing measure for the limit of $\Pp_n$. \medskip

	We turn our attention to the displacements in our model. Consider yet another family $\{ X_y \: : \: y \in \U\setminus \{ \emptyset\} \}$ of
	iid random variables. Each $X_y$ represents the  shift of the particle $y$ from its place of birth. Therefore the position $V(x)$ of the particle
	$x \in \T$ can be represented via
	\begin{equation*}
 		V(x) = \sum_{y \in (\emptyset, x]} X_y,
 	\end{equation*}
 	where $(\emptyset, x] = \{ y \in \T \: : \: y\leq x, \: y \not = \emptyset \}$. The position of the rightmost particle
	\begin{equation}\label{eq:2:max}
 		M_n = \max_{|x|=n}V(x)
 	\end{equation}
 	is the maximum of $Z_n$ correlated random walks with step distribution $X$.
 	In order to be able to analyse different order statistics of
	$\{ V(x) \}_{ |x|=n}$, one is lead to study the point process
 	\begin{equation*}
 		\Pp_n = \sum_{|x|=n}\delta_{V(x)}.
 	\end{equation*}
	In what follows $\Pp_n$ is treated as a random element of the space of non-negative Radon measures (see  Chapter~3 in \cite{resnick2008extreme}
	and section~6.1.3 in \cite{resnick2007heavy}). Let $\bR= (- \infty, \infty]$ be a homomorphic image of $(0,1]$. Denote by $\M(\bR)$
	the space of measures on $\bR$ which charge finite mass to each of the compact subsets of $\bR$.
	Let $C_c^+(\bR)$ denote the space of non-negative bounded continuous functions which have compact support.
	We say that a sequence of point measures $\{ \nu_n \}_{n \in \NN}$ converges vaguely to a measure $\nu$ on $\bR$ if
	$\lim_{n \to \infty} \nu_n(f) = \nu(f)$ where
	$\nu(f) = \int f \ud \nu$ for all $f \in C_c^+(\bR)$.
	According to Proposition~3.17 in \cite{resnick2008extreme}, $\M(\bR)$ equipped with the vague topology is a Polish space.

 	We will study the weak convergence of $\Pp_n$ as a random element of $\M(\bR)$.
	As it turns out the weak convergence in  $\M(\bR)$ can be put in relatively simple terms.
	By~\cite[Proposition 3.19]{resnick2008extreme},
	for random measures $\bsymb{\Theta}_n$, $\bsymb{\Theta}$,
	\begin{equation*}
		\bsymb{\Theta}_n  \Rightarrow \bsymb{\Theta} \quad \mbox{in $\M(\bR)$}
	\end{equation*}
	if and only if
	\begin{equation*}
		\bsymb{\Theta}_n(f) = \int f(x) \: \bsymb{\Theta}_n(\ud x) \Rightarrow
		\bsymb{\Theta}(f) = \int f(x) \: \bsymb{\Theta}(\ud x) \quad \mbox{in $\R$}
	\end{equation*}
	for any $f$ from the class $C_c^+(\bR)$. Here
	$\Rightarrow$ represents a weak convergence of random elements of a given space. The second condition is indeed
	less complex than the first one since for $f \in C_c^+(\bR)$, the integral $\int f(x) \: \bsymb{\Theta}_n(\ud x) $
	is a bona fide random variable. Since weak convergence in $\RR$ can be tested using Laplace functionals, and since $C_c^+(\bR)$ is
	closed under scaling by a positive factor, we conclude that a sequence of point processes $\bsymb{\Theta}_n$ converges weakly to
	$\bsymb{\Theta}$ in $\M(\bR)$ if
	\begin{equation*}
		\lim_{n \to \infty} \EE[ \exp \{ - \bsymb{\Theta}_n(f) \}] = \EE[ \exp \{ -\bsymb{\Theta}(f)\}]
	\end{equation*}
	for all $f \in C_c^+(\bR)$.
	We shall use Laplace functional to obtain the weak limit of $\Pp_n$.\medskip

	Aiming to provide centring and scaling for $\Pp_n$ we define the shift and scaling for point measures as follows.
	For $\nu \in \M(\bR)$ given via $\nu =\sum_i \delta_{x_i}$ let the the shift by $x \in \RR$ and scaling by $y >0$ be defined as
	\begin{equation*}
		\shift_x \nu = \sum_{i}\delta_{x_i+x}, \qquad \scale_y \nu = \sum_{i}\delta_{y x_i}.
	\end{equation*}
	Our goal is to provide sequences $\{a_n\}_{n \in \NN}$ and $\{b_n\}_{n \in \NN}$ such that
	\begin{equation*}
		\scale_{a^{-1}_n}\shift_{-b_n}\Pp_n = \sum_{|x| = n} \delta_{(V(x)-b_n)/a_n}
	\end{equation*}
	converges weakly in $\M(\bR)$ to a non-trivial limit. By the choice of the space, this will provide a description of the positions of the particles
	in the $n-$th generation in the vicinity of the rightmost one. \medskip

	If the displacements have thin tails, that is~\eqref{eq:1:thin} holds, then under some mild technical assumptions~\cite{madule2017} there are
	constants $c_1$, $c_2$ such that
	\begin{equation}\label{eq:2:coffee}
		\shift_{-c_1n-c_2\log n}\Pp_n
	\end{equation}
	converges to a randomly shifted decorated Poisson process~\cite{subag2015freezing}.

	\begin{df}\label{df:2:sdppp}
		A point process $\bsymb{\Theta}$ is called a randomly shifted decorated
		Poisson process if there exists a Radon measure $\mu$, a point process $\bsymb{\Lambda}$ and a random variable $S$ such that
		\begin{equation*}
			\bsymb{\Theta} = \sum_{k} \shift_{\xi_k+S}\bsymb{\Lambda}_k,
		\end{equation*}
		where $\sum_k\delta_{\xi_k}$ is a Poisson point process with intensity $\mu$, $\{\bsymb{\Lambda}_k\}_{k \in \NN}$ are iid copies of
		$\bsymb{\Lambda}$ such that $\sum_k\delta_{\xi_k}$, $\{\bsymb{\Lambda}_k\}_{k \in \NN}$ and $S$ are independent. In this case one writes
		$\bsymb{\Theta} \sim {\rm SDPPP}(\mu, \bsymb{\Lambda}, S )$.
	\end{df}

	It turns out that the weak limit of~\eqref{eq:2:coffee} is  ${\rm SDPPP}(e^{-x}\ud x, \bsymb{\Lambda}, c_3D_{\infty} )$ for some point
	process $\bsymb{\Lambda}$ and some constant $c_3$, see~\cite{madule2017} for details. Here $D_\infty$ is the aforementioned limit
	of so-called derivative martingale associated with BRW.\medskip

	If $L$ is regularly varying with index $r \in (0,1)$, i.e.~\eqref{eq:1:rv} holds then, under some additional regularity assumptions on $L$,
	there exist slowly varying functions $\ell_1$ and $\ell_2$,
	\begin{equation}\label{eq:2:pointstretched}
		\scale_{n^{1-1/r}\ell_2(n)^{-1}}\shift_{-n^{1/r}\ell_1(n)}\Pp_n
	\end{equation}
	converges in law to a cluster Cox process~\cite{dyszewski2022extremal}. To provide a detailed description of the limit let
	$\{A_k\}_{k\in \NN}$ be a collection of iid random variables distributed as
	\begin{equation*}
		\PP[A_k = j] = \frac 1v\sum_{l=0}^\infty m^{-l}\PP[Z_l =j],  \qquad j \in \NN,
	\end{equation*}
	where $v$ is the normalising constant
	\begin{equation*}
		v = \sum_{l=0}^\infty m^{-l}\PP[Z_l >0].
	\end{equation*}
	Let $\{ \ell_k\}_k$ be a Poisson point process on $\RR$ with intensity $e^{-x}\ud x$. The limit of~\eqref{eq:2:pointstretched} can be written as
	\begin{equation}\label{eq:2:defV}
		\Pp = \sum_{k} A_k\delta_{\ell_k - \log(vW)}
	\end{equation}
	Then $\Pp$ is a cluster Cox process with Laplace functional given via
	\begin{equation*}
		f \mapsto \EE\left[ \exp \left\{ - \sum_{l=0}^\infty m^{-l} W \int \left(1-e^{-f(x)Z_l} \right) e^{-x}\ud x \right\} \right]
	\end{equation*}
	for measurable $f \colon \RR \to [0, +\infty)$. It should be noted that $\Pp$ given in~\eqref{eq:2:defV} is
	${\rm SDPPP}(e^{-x}\ud x, A_1\delta_0, \log(v W) )$.\medskip

	We now turn to our first main result. To provide the normalization for $\Pp_n$, we need to introduce the conditions imposed on the
	step distribution. We will work in two cases namely sub- and sup-logarithmic regimes. We begin with the latter, since in this case
	the classical extreme value theory applies.

	\begin{assumption}\label{as:step1}
		We assume that the random variable $X$ has distribution of the form
		\begin{equation*}
			\PP[X>t] = a(t)e^{-L(t)}, \qquad t \in \R,
		\end{equation*}
		where $a(t) \to a>0$,  $L(\cdot)$ is $C^2$ function such that $L'$ is slowly varying and
		\begin{equation}\label{conditions-on-L}
			\lim_{x \to \infty} \frac{L''(x)}{L'(x)^2} =0, \qquad \lim_{x \to \infty} \frac{L(x)}{xL'(x) \sqrt{\log\log(L(x))}} = \infty
		\end{equation}
		and the function $x \mapsto x^{-1/3}L(x)$ is eventually decreasing. Assume moreover the following left tail condition
		\begin{equation*}
			\PP[X<-t] \leq t^{-\varepsilon}.
		\end{equation*}
		for some $\varepsilon>0$ and sufficiently large $t>0$.
	\end{assumption}%

	We note that the appearance of $\sqrt{\log\log(L(x))}$ in the second condition in \eqref{conditions-on-L} may be an artefact of our proof.
	The first condition in \eqref{conditions-on-L}
	implies in particular that $L$ grows faster than logarithmic.
	Assumption~\ref{as:step1} ensures that $e^{-L(x)}$ is a von-Mises function.
	This in turn allows us to treat $\Pp_n$
	using extreme value theory. In this spirit denote
	\begin{equation}\label{eq:2:seq}
		b_n = \inf \{ t \geq 0 \: :\: \PP[X>t]\leq m^{-n} \}, \qquad a_n = 1/L'(b_n).
	\end{equation}
	Note that under Assumption~\ref{as:step1} both $b_n$ and $a_n$ grow faster than any polynomial. For example, in the case of lognormal
	displacements where
	\begin{equation*}
		L(t) = ((\log t)^2+2\log\log t)/2,
	\end{equation*}
	one has
	\begin{multline*}
		b_n = \exp\left\{ \sqrt{2n\log m-2\log(2n\log m)}\right\}(1+o(1))\qquad  \\
			\mbox{and} \qquad a_n = b_n/\sqrt{2n \log m}(1+o(1)).
	\end{multline*}
	Taking $b_n$ and $a_n$ as in~\eqref{eq:2:seq} gives
	\begin{equation*}
		m^n \PP[X > b_n+a_n x] \to e^{-x}
	\end{equation*}
	as $n \to \infty$ for any $x \in \RR$ and furthermore the maximum of $m^n$ independent copies of $(X-b_n)/a_n$ converges in law to the
	Gumbel law. Note that at this point one uses the limiting relations in Assumption~\ref{as:step1}. We will use this limit theorem in the proof of
	our first result.
	
	\begin{thm}[The suplogarithmic case]\label{thm:2:1}
		Under Assumptions~\ref{as:galton} and \ref{as:step1}, one has for the sequences $\{a_n\}_{n \in \NN}$ and $\{b_n\}_{n \in \NN}$ given
		in~\eqref{eq:2:seq} the convergence
		\begin{equation*}
			\scale_{a^{-1}_n}\shift_{-b_n}\Pp_n \Rightarrow \Pp
		\end{equation*}
		in $\M(\bR)$, where $\Pp$ is defined in~\eqref{eq:2:defV}.
	\end{thm}

	Stating results in terms of point process convergence allows to treat joint convergence of the extremes of
	$\{V(x)\}_{|x|=n}$ by standard arguments~\cite[Section 4.6]{bhattacharya:2017:point}. We mention in passing that
	under Assumptions~\ref{as:galton} and \ref{as:step1}, for the sequences $\{a_n\}_{n \in \NN}$ and $\{b_n\}_{n \in \NN}$ given
	in~\eqref{eq:2:seq} and $M_n$ given via~\eqref{eq:2:max},
	\begin{equation*}
		\PP \left[ M_n \leq b_n+a_nx \right] = \PP \left[ \scale_{a^{-1}_n}\shift_{-b_n}\Pp_n(x, +\infty] =0 \right] \to
		\EE \left[ e^{- v W e^{-x} } \right]
	\end{equation*}
	for any real $x$. \medskip

	In the proof of Theorem~\ref{thm:2:1} given in Section~\ref{sec:4:proofs} we use the principle of one big jump to approximate $\Pp_n$ via
	\begin{equation*}
		 \Tp_n = \sum_{|x|=n} \delta_{T(x)}, \quad T(x) = \max_{y \in (\emptyset, x]} X_y.
	\end{equation*}
	Next, we use a stopping line argument. We consider
	\begin{equation*}
		\mathcal{T}_n =
		\left\{ |v|\leq n \: : \: \exists |x|=n, \: x \geq v, \: \max_{y<v}X_y \ll b_n, \: X_v \approx b_n\right\},
	\end{equation*}
	with a rigorous definition given in~\eqref{eq:4:line}. One then approximates $\Tp_n$ via
	\begin{equation*}
		\sum_{v \in \mathcal{T}_n}E_{n-|v|}(v)\delta_{X_v}, \qquad
		E_n(v) =
		\# \left\{ |x|=n+|v| \: : \:  x \geq v \right\}.
	\end{equation*}
 	The main advantage of the above representation is that
	$\{E_{n-|v|}\}_{v \in \mathcal{T}_n}$ are independent given $\mathcal{T}_n$. Now one shows two facts. Firstly,
	\begin{equation*}
		\sum_{v \in \mathcal{T}_n}\delta_{(X_v-b_n)/a_n} \Rightarrow  \sum_{k} \delta_{\ell_k - \log(vW)},
	\end{equation*}
	where  $\sum_{k} \delta_{\ell_k}$ is a Poisson point process with intensity $e^{-x}\ud x$. Secondly
	\begin{equation*}
		\{ E_{n-|v|}(v) \}_{v \in \mathcal{T}_n} \Rightarrow \{A_k\}_{k \in \NN}.
	\end{equation*}
	This in turn allows us to infer that $\scale_{a^{-1}_n}\shift_{-b_n}\Tp_n \Rightarrow \Pp$ which further implies
	$\scale_{a^{-1}_n}\shift_{-b_n}\Pp_n \Rightarrow \Pp$.
	To the best of our knowledge this is a first instance of an application of stopping-line argument in the study of BRW with
	heavy tailed displacements. We believe that this approach is robust and allows for an efficient analysis of BRW in the presence of
	heavy tails in other cases. \medskip

	Going back to the review of the literature, we mention in passing that
	in the case when $L$ is logarithmic, say $L(t) = \alpha \log t$ for $\alpha, t> 0$,
	it is known~\cite{bhattacharya:2017:point}
	that
	\begin{equation*}
		\scale_{m^{-n/\alpha}}\Pp_n \Rightarrow \sum_{k} A_k\delta_{(vW)^{1/\alpha}\xi_k },
	\end{equation*}
	where $\sum_{k}\delta_{\xi_k}$ is a Poisson point process with intensity $\ind_{\{x \geq 0\}} \alpha x^{-\alpha-1}\ud x$.

	\begin{assumption}\label{as:step2}
		We assume that the random variable $X$ is non-negative and has distribution of the form
		\begin{equation*}
			\PP[X>t] = e^{-L(t)}, \qquad t \in \R,
		\end{equation*}
		where $L(\cdot)$ is slowly varying such that $L(x) = o(\log x)$ as $x \to \infty$.
	\end{assumption}%

	To study $\Pp_n$ under the above conditions we need to use a non-linear scale. For $\nu \in \M(\bR)$ given via $\nu =\sum_i \delta_{x_i}$ let
	\begin{equation*}
		\sL \nu = \sum_{i}\delta_{L(x_i)}.
	\end{equation*}
	\begin{thm}[The sublogarithmic case]\label{thm:L:sublog}
		Under the Assumptions~\ref{as:galton} and \ref{as:step2}
		\begin{equation*}
			\shift_{-n \log m}\sL\Pp_n \Rightarrow \Pp
		\end{equation*}
		in $\M(\bR)$ where $\Pp$ is the point process given in~\eqref{eq:2:defV}.
	\end{thm}

	As it was the case for Theorem~\ref{thm:2:1}, we may infer limit theorems for the extremes.
	For example under Assumptions~\ref{as:galton} and \ref{as:step2}, for $M_n$ given via~\eqref{eq:2:max},
	\begin{equation*}
		\PP \left[ L(M_n) \leq n\log m +x \right] = \PP \left[\shift_{-n\log m}\sL\Pp_n(x, +\infty] =0 \right] \to
		\EE \left[ e^{- v W e^{-x} } \right]
	\end{equation*}
	for any real $x$. The above result goes in line with limit theorem for random walks with increments having log-slowly varying
	tails where one has to apply a non-linear scale~\cite[Theorem 4.1]{darling1952influence}. \medskip

	The proof of Theorem~\ref{thm:L:sublog}
	goes along the same lines as the arguments used for Theorem ~\ref{thm:2:1}. It turns out that the stopping line argument
	can be used in this case as well with the stopping line of the form
	\begin{multline*}
		\mathcal{S}_n =
		\Big\{ |v|\leq n \: : \: \exists |x|=n, \: x \geq v, \\ \max_{y<v}L(X_y) \ll n \log m , \: L(X_v) \approx  n\log m \Big\}.
	\end{multline*}

	We finally mention that the limiting point process $\Pp$ given in~\eqref{eq:2:defV} as a randomly shifted decorated Poisson process is
	superposable \cite[Section 3.2]{brunet:derrida:2011}, which
	means that a union of independent copies of $\Pp$ when view from the point of the rightmost particle has the same law as $\Pp$ viewed from
	the position of the rightmost particle.
	As in~\cite[Section 3]{dyszewski2022extremal}
	this can be also seen directly by appealing to~\eqref{eq:2:ST} which implies that
	\begin{equation*}
		\sum_{k=1}^{Z_1} \Pp^{(k)} \stackrel{d}{=} \shift_{\log m}\Pp,
	\end{equation*}
	where $\Pp^{(k)}$ are iid copies of $\Pp$ independent of $Z_1$. If we thus define
	\begin{equation*}
		\tip \sum_{i}\delta_{x_i} =
		\left\{ \begin{array}{cc}\sum_i\delta_{x_i-\max_j x_j} & \mbox{if } \max_jx_j<\infty \\ o & \mbox{otherwise}\end{array} \right.,
	\end{equation*}
	where $o$ denotes the null measure, we see that
	\begin{equation*}
		\tip\sum_{k=1}^{Z_1} \Pp^{(k)} \stackrel{d}{=} \tip\Pp.
	\end{equation*}
	Furthermore, if the branching is deterministic $\Pp$ is exponentially stable~\cite{maillard:2013} (see \cite[Remark 3.5]{dyszewski2022extremal}
	for details).

\section{Random walk estimates}\label{sec:3:RWestimates}

	In this section we will present a couple of large deviation estimates for random walks with steps distributed as in
	Assumption~\ref{as:step1} or Assumption~\ref{as:step2}.
	We begin with some auxiliary estimates on the truncated exponential moments for random variables with logarithmically slowly varying tails.
	These estimates will be later used to establish deviation estimates for the corresponding random walk.
	In the next two lemmas we assume that
	\begin{equation*}
		\PP[X>t] = e^{-L(t),}
	\end{equation*}
	for a slowly varying function $L(\cdot)$.

\begin{lem}\label{lem:3:trunk}
	Suppose that for some $\xi \in (0,1)$, the function $t \mapsto t^{-\xi}L(t)$ is eventually decreasing. Then for any $\gamma \in (0,1)$ and
	$y>0$ one has
	\begin{equation*}
		\EE\left[e^{\gamma L(y) X/y}\ind_{\{ X \leq y\}}\right] \leq
		1 + (1+o(1))L(y)^{(1-1/\xi)/2}.
	\end{equation*}	
\end{lem}

\begin{proof}
	Write $s= \gamma L(y)/y$. Firstly note that we have a simple bound
	\begin{equation*}
		\EE\left[e^{s X}\ind_{\{ X < 0\}}\right] \leq \PP \left[ X<0 \right].
	\end{equation*}
	Secondly, by integrating by parts we get
	\begin{multline*}
		\EE\left[e^{sX}\ind_{\{X \in [0,y]\}}\right] =  \int_0^y se^{st}\PP[X>t]\ud t + \PP[X \geq 0] - e^{sy}\PP[X>y] \\
		\leq \int_0^y se^{st}\PP[X>t]\ud t + \PP[X \geq 0] .
	\end{multline*}
	For the integral we have
	\begin{equation*}
		\int_0^y se^{st}\PP[X>t]\ud t = \int_0^1 \gamma L(y) e^{\gamma t L(y) - L(ty)}\ud t.
	\end{equation*}
	Put $\varepsilon = (1/\xi-1)/2$ and
	take $m_y=L(y)^{-1-\epsilon}$ and $M_y = L(y/L(y)^{1+\varepsilon})/L(y)$.
	We can write the integral on the right-hand side of the last display as a sum of integrals over the intervals $(0, m_y)$, $[m_y, M_y)$ and $[M_y,1)$.
	The first one is easily
	bounded via $(1+o(1))L(y)^{-\varepsilon}$. The second one can be estimated by
	\begin{equation*}
		L(y) \exp \{ (\gamma -1) L(y/L(y)^{1+\varepsilon}) \}.
	\end{equation*}
	Note that under the proviso concerning the monotonicity, for sufficiently large $y$,
	\begin{equation*}
		L(y/L(y)^{1+\varepsilon}) \geq L(y)^{(1- \xi)/2}.
	\end{equation*}
	Lastly, for the third interval we have, using the above inequality in combination with the estimate $L(ty) \geq t^\xi L(y)\geq tL(y)$ for
	$t \in (0,1)$ we get
	\begin{multline*}
		\int_{M_y}^1 L(y) e^{\gamma t L(y) - L(ty)}\ud t \leq \int_{M_y}^1 L(y) e^{(\gamma -1)t L(y)}\ud t \\
			\leq (1-\gamma)^{-1}e^{(\gamma-1)L(y)^{(1- \xi)/2}}.
	\end{multline*}
	Combining all the above estimates yields our claim.
\end{proof}

\begin{lem}\label{lem:3:tree}
	Let Assumption~\ref{as:step1} be in force.
	Then for any $\gamma \in (0,1)$ there exists $x$ sufficiently large such that for any $y> x/2$ and $z\leq x/2$ one has
	\begin{multline*}
		\EE\left[e^{ \gamma L(y) X/y}\ind_{\{ X \in (y,x-z]\}}\right] \leq \\ \const e^{-L(y)} e^{(\gamma L(y)/y -L'(y)/3)(x-z) + yL'(y)/3} +  e^{(\gamma-1) L(y)}.
	\end{multline*}	
\end{lem}
\begin{proof}
	Put $s= \gamma L(y)/y$.
	By yet another appeal to the integration by parts formula,
	\begin{multline*}
		\EE\left[e^{sX}\ind_{X \in (y,x-z]}\right] =  \int_y^{x-z} se^{st}\PP[X>t]\ud t + e^{sy}\PP[X>y]  + \\- e^{s(x-z)}\PP[X>x-z]
		\leq  \int_y^{x-z} se^{st}\PP[X>t]\ud t + e^{sy}\PP[X>y] .
	\end{multline*}
	The second term present in the last display is equal to $\exp\{(\gamma-1)L(y) \}$. To estimate the integral first write
	\begin{equation*}
		\int_y^{x-z} se^{st}\PP[X>t]\ud t = \int_{y}^{x-z} \gamma \frac{L(y)}y e^{\gamma L(y)t/y - L(t)}\ud t.
	\end{equation*}
	By the mean value theorem and regular variation of $L'$, for sufficiently large $x$,
	\begin{equation*}
		L(t) - L(y) = (t-y)L'(\theta_t) \geq (t-y)L'(y)/3, \qquad t \in (y,x-z), \: \theta_t \in (y,t).
	\end{equation*}
	Therefore
	\begin{equation*}
		\int_{y}^{x-z} \gamma \frac{L(y)}y e^{\gamma L(y)t/y - L(t)}\ud t \leq \const e^{-L(y)} e^{(\gamma L(y)/y -L'(y)/3)(x-z) + yL'(y)/3}.
	\end{equation*}
\end{proof}

\subsection{The suplogarithmic case}

Let $\{ S_n \}_{n \in \NN}$ be a random walk generated by $X$. That is, $S_0=0$ and for $n \geq 1$, $S_n = X_1+X_2+\ldots +X_n$, where
$\{X_k\}_{k \in \NN}$ are iid copies of $X$ with a law satisfying Assumption \ref{as:step1}. We will denote also $N_n = \max_{k \leq n}X_k$.
In what follows we suppose that the Assumption~\ref{as:step1} is in force.  Write
\begin{equation*}
	x_n=x_n(K) = b_n+Ka_n, \qquad y_n = (1-\delta) b_n,
\end{equation*}
where $K \in \RR$, $\delta \in (0,1)$ are fixed and
\begin{equation*}
	b_n = \inf \{ t \geq 0 \: :\: \PP[X>t]\leq m^{-n} \}, \qquad a_n = 1/L'(b_n).
\end{equation*}

\begin{lem}\label{lem:3:small}
	Suppose that for some $\xi \in (0,1)$, the function $t \mapsto t^{-\xi}L(t)$ is eventually decreasing.
	Then, for any $K\in \RR$ and $\delta\in (0,1)$,
	\begin{equation*}
		\PP[S_n>x_n, \: N_n \leq y_n] = o\left(m^{-n}\right).
	\end{equation*}
\end{lem}
\begin{proof}
	Take $s = \gamma L(y_n)/y_n$ for some $\gamma \in (0,1)$ and consider the following upper bound
	\begin{equation*}
		\PP[S_n>x_n, \: N_n \leq y_n]  \leq e^{-sx_n} \EE\left[e^{sX}\ind_{\{ X \leq y_n\}}\right]^n.
	\end{equation*}
	An appeal to Lemma~\ref{lem:3:trunk} yields that for some universal constant $\const $,
	\begin{equation}\label{eq:3:cake}
 		\EE\left[e^{sX}\ind_{\{ X \leq y_n\}}\right] \leq 1 + \const n^{(1-1/\xi)/2}.
	\end{equation}
	Using the inequality $1-x \leq e^{-x}$ valid for all $x$, we can conclude that
	\begin{equation}\label{eq:3:thesame}
		\PP[S_n>x_n, \: N_n \leq y_n]  \leq \exp \left\{ -\gamma L(y_n) x_n/y_n + C n^{(3-1/\xi)/2} \right\}.
	\end{equation}
	Note that
	\begin{equation*}
		\lim_{n \to \infty} \gamma (L(y_n) x_n)/(y_nL(b_n)) = \gamma/(1-\delta)
	\end{equation*}
	so taking $\gamma$ sufficiently close to $1$ we can get for sufficiently large $n$,
	\begin{equation*}
		\PP[S_n>x_n, \: N_n \leq y_n]  \leq \exp \left\{ - n \log m /(1-\delta/2) + C n^{(3-1/\xi)/2} \right\}
	\end{equation*}
	which yields our claim.
\end{proof}

Let
\begin{equation*}
	z_n =  \frac{Tb_n\log n}{L(b_n)}
\end{equation*}
for some (sufficiently large) constant $T>0$ that will depend on $\delta \in (0,1)$. Note that under Assumption~\ref{as:step1}, $z_n =o(b_n)$ and
$a_n = o(z_n)$.

\begin{lem}\label{lem:3:small2}
	Let the Assumption~\ref{as:step1} be in force.
	Then there exists sufficiently small $\delta \in (0,1)$ and sufficiently large $T>0$ such that for any value of $K \in \RR$,
	\begin{equation*}
		\PP[S_{n-1}+X>x_n, \: N_{n-1} \leq y_n, \: X\in (y_n,x_n-z_n]] = o\left(nm^{-n}\right).
	\end{equation*}
\end{lem}
\begin{proof}
	Use a simple estimate
	\begin{multline*}
		\PP[S_{n-1}+X>x_n, \: N_{n-1} \leq y_n, \: X\in (y_n,x_n-z_n]] \leq  \\
		\exp\{ -\gamma L(y_n) x_n/y_n \} \EE\left[e^{ \gamma L(y_n) X/y_n}\ind_{\{ X \in (y_n,x_n-z_n]\}}\right]
		\EE\left[e^{sX}\ind_{\{ X \leq y_n\}}\right]^n.
	\end{multline*}
	The last factor was already bounded in the proof of Lemma~\ref{lem:3:small}. To bound the second one we appeal to
	Lemma~\ref{lem:3:tree}.
	Using the fact that
	\begin{equation*}
		L(y_n) = L(b_n) - \delta b_n L'(\theta_n).
	\end{equation*}
	for some $\theta_n \in ((1-\delta)b_n, x_n)$ combined with regular variation of $L'$
	we arrive at
	\begin{multline*}
		\PP[S_{n-1}+X>x_n, \: N_{n-1} \leq y_n, \: X\in (y_n,x_n-z_n]] \leq  \\
		\const\exp\left\{ - L(b_n) +  \left( \delta/(1-\delta) -(1+\delta)/3 \right) b_nL'(b_n) - \gamma L(y_n)z_n/(2y_n)  \right\}.
	\end{multline*}
	If $\delta$ is small enough the last bound is smaller than
	\begin{equation*}
		\const\exp\left\{ - L(b_n) - \gamma L(y_n)z_n/(2y_n)  \right\}
	\end{equation*}
	which is $o(nm^{-n})$ provided that $T$ is sufficiently large.
	The second factor can be treated similarly as $S_n$ in the proof of Lemma~\ref{lem:3:small}.
	This implies the claim.
\end{proof}

\section{Proofs}\label{sec:4:proofs}
	In this section we will present the proofs of our main results. Both arguments differ slightly in technical aspects, but they rely on the same idea.
	Namely, we will approximate $\Pp_n$ via the point process
	\begin{equation*}
		 \Tp_n = \sum_{|x|=n} \delta_{T(x)}, \quad T(x) = \max_{y \in (\emptyset, x]} X_y.
	\end{equation*}
	Next we will show that $\Tp_n$ has the desired limit.

\subsection{The suplogarithmic case}
	We will now present the proof of Theorem~\ref{thm:2:1}. Throughout the subsection Assumption~\ref{as:galton} and Assumption~\ref{as:step1}
	are in force.
	Recall that
	\begin{equation*}
		b_n = \inf \{ t \geq 0 \: :\: \PP[X>t]\leq m^{-n} \},  \qquad z_n =  \frac{Tb_n\log n}{L(b_n)}
	\end{equation*}
	and
	\begin{equation*}
		y_n = (1-\delta)b_n, \quad a_n =1/L'(b_n), \qquad x_n = x_n(K) = b_n +Ka_n.
	\end{equation*}	
	We will first consider the following subsets of the set of particles in the $n-$th generation $\mathbb{T}_n$,
	\begin{align*}
		\mathcal{A}_n^{(1)} & = \{ x \in \mathbb{T}_n \ : \: T(x) \leq y_n \}, \\
		\mathcal{A}_n^{(2)} & = \{ x \in \mathbb{T}_n \ : \:  \exists v_1, v_2 \leq x, \: v_1 \neq v_2, \:  \min \{ X_{v_1}, X_{v_2} \}\geq y_n \}, \\
		\mathcal{A}_n^{(3)} & = \{ x \in \mathbb{T}_n\setminus (\mathcal{A}_n^{(1)} \cup \mathcal{A}_n^{(2)} ) \ : \: T(x) \leq b_n -z_n\}.
	\end{align*}
	We aim to show that the particles from $\mathcal{A}_n= \mathcal{A}_n^{(1)} \cup \mathcal{A}_n^{(2)} \cup \mathcal{A}_n^{(3)}$ do not contribute to the limit.
	\begin{lem}\label{lem:4:ao}
		Let
		\begin{equation*}
			 \Pp_n^{\mathcal{A}_n} = \sum_{x \in \mathcal{A}_n} \delta_{V(x)}.
		\end{equation*}
		Then under the Assumptions~\ref{as:galton} and \ref{as:step1},
		\begin{equation*}
			 \scale_{a^{-1}_n}\shift_{-b_n}\Pp_n^{\mathcal{A}_n}  \Rightarrow o,
		\end{equation*}
		weakly in $\M(\bR)$, where $o$ denotes the null measure.
	\end{lem}

	\begin{proof}
		It suffices to show that for any $K \in \RR$,
		 \begin{equation*}
		 	\lim_{n \to \infty}\Pp_n^{\mathcal{A}_n}[[x_n(K), +\infty]] =0
		 \end{equation*}
		 in probability. We will appeal to the decomposition $\Acal_n = \Acal_n^{(1)} \cup \Acal_n^{(2)} \cup \Acal_n^{(3)}$, and treat
		 $\mathcal{A}_n^{(j)}$'s one by one.
		For the first one write
		\begin{equation*}
			\PP[\exists x \in \mathcal{A}_n^{(1)} \: : \: V(x) > x_n] \leq m^n \PP[S_n > x_n, \: N_n\leq y_n]
		\end{equation*}
		Invoke Lemma~\ref{lem:3:small} to get
		\begin{equation*}
			\PP[S_n > x_n, \: N_n\leq y_n]  = o(m^{-n}).
		\end{equation*}
		To treat $\mathcal{A}_n^{(2)}$ just note that
		\begin{equation*}
			\PP[\mathcal{A}_n^{(2)} \neq \emptyset] \leq m^n e^{-2 L(\delta b_n)} = o(1).
		\end{equation*}
		Finally, for $\mathcal{A}_n^{(3)}$ we argue as for  $\mathcal{A}_n^{(2)}$ but instead of using Lemma~\ref{lem:3:small} we use 	
		Lemma~\ref{lem:3:small2}.
	\end{proof}

	\begin{lem}\label{lem:4:tea}
		Suppose that the Assumptions~\ref{as:galton} and \ref{as:step1} are in force.
		Let
		\begin{equation*}
			M_n^{(A)} = \max \left\{ \left|V(x)-T(x) \right| \: : \:
			x \in \mathbb{T}_n\setminus \mathcal{A}_n\right\}.
		\end{equation*}
		Then, for any $\delta\in (0,1)$ and $T>0$, we have
		\begin{equation*}
			\lim_{n \to \infty}M_n^{(A)}/a_n =0
		\end{equation*}
		in probability.
	\end{lem}
	\begin{proof}
		Take $\varepsilon>0$ and write
		\begin{equation*}
			\PP\left[M_n^{(A)} > \varepsilon a_n\right]
			 \leq n m^n \PP[X > b_n-z_n] \PP[|S_{n-1}| > \varepsilon a_n].
		\end{equation*}
		The logarithm of the second term can be expressed in the following fashion
		\begin{equation*}
			\log \PP[X>b_n-z_n] = L(b_n) -z_nL'(\theta_n)
		\end{equation*}
		for some $\theta_n \in (b_n-z_n, b_n)$.  Since $z_n L'(\theta_n) = o(\log n)$
		it is sufficient to consider a generous upper bound
		\begin{equation*}
 			\PP[|S_{n-1}| > \varepsilon a_n] \leq n \PP[|X_1| > \varepsilon a_n/n] \leq n (\varepsilon a_n/n)^{-\varepsilon}.
		\end{equation*}	
		Since $a_n$ grows faster than any polynomial this constitutes
		\begin{equation*}
			\PP\left[M_n^{(A)} > \varepsilon a_n\right] \leq  C/n
		\end{equation*}
		and secures the claim.
	\end{proof}

	Using the above lemma we can approximate $\Pp_n$ via $\Tp_n$.We will present the analysis of the latter followed by an approximation lemma.
	Consider the stopping line
	\begin{multline}\label{eq:4:line}
		\mathcal{T}_n =
		\Big\{ |v|\leq n \: : \: \exists x \in \mathbb{T}_n\setminus \mathcal{A}_n, \: x \geq v,\\ \: \max_{y<v}X_y \leq b_n-z_n, \: X_v> b_n -z_n \Big\}.
	\end{multline}
	For $v \in \mathbb{U}$ and $n \in \NN$ consider
	\begin{equation*}
		E_n(v) =
		\# \left\{ |x|=n+|v| \: : \:  x \geq v \right\}.
	\end{equation*}
	Then for any $f \in C_c^+((-\infty, \infty])$, for sufficiently large $n$, on the event $\{\mathcal{A}_n^{(2)}=\emptyset\}$,
	\begin{equation*}
		\sum_{|x|=n} f\left(\frac{T(x)-b_n}{a_n}\right) =\sum_{v \in \mathcal{T}_n} f\left(\frac{X_v -b_n}{a_n}\right) E_{n-|v|}(v).
	\end{equation*}
	Note that $E_{n - |v|}(v)$ denotes the number of children of the vertex $v$ at the $n-$th generation
	The sum on the right can be linked to yet another point process %$\widehat{\Pp}_n$
on $(-\infty, \infty]\times \NN$ given via
    	\begin{equation*}
      		\Npa_n = \sum_{x \in \mathcal{T}_n} \delta_{(X_x-b_n)/a_n}\otimes \delta_{n-|x|}.
    	\end{equation*}
%	The main feature of $\widehat{\Pp}_n$ is that given $\mathcal{T}_n$, the random variables $\{E_{n-|x|}(x) \}_{x \in \mathcal{T}_n}$
%	are independent and are independent from $\{X_{x}\}_{x \in \mathcal{T}_n}$ given $\mathcal{T}_n$. Furthermore
%	\begin{equation*}
%		\PP[ (X_x-b_n)/a_n \in \cdot | x \in \mathcal{T}_n] = \PP[ (X-b_n)/a_n \in \cdot | X> b_n-z_n]
%	\end{equation*}
%	and
%	\begin{equation*}
%		\PP[ E_{n-|x|}(x) \in \cdot | x \in \mathcal{T}_n] = \PP[Z_{n-k} \in \cdot | Z_{n-k}>0] \quad \mbox{for $|x|=k$}.
%	\end{equation*}
%	With these observations at hand we can state and prove a limiting result for $\widehat{\Pp}_n$.
	To note a limiting result for the latter define a measure on $\NN$ via
	\begin{equation*}
		\rho(\cdot)  = \sum_{j=0}^\infty m^{-j}\delta_j(\cdot)
	\end{equation*}
	and a Poisson point process $\Npa$ on $\RR \times \NN$ given via
	\begin{equation*}
		\Npa = \sum_{k}\delta_{\ell_k - \log(aW)}\otimes \delta_{\iota_k},
	\end{equation*}
	where $\{ (\ell_k, \iota_k)\}_{k}$ is a Poisson point process
	with intensity $e^{-x}\ud x \rho(\ud j)$. Then $\Npa$ has a Laplace functional of the form
	\begin{equation*}
		\EE \left[e^{ - \int f(x,j) \Npa(\ud x, \ud j) } \right] = \EE \left[\exp\left\{ - W \int \left(1-e^{-f(x,l)}\right) e^{-x}\ud x \rho(\ud l) \right\}\right].
	\end{equation*}

	\begin{prop}\label{prop:4:banana}
		Under the Assumption~\ref{as:galton} and Assumption~\ref{as:step1}, we have
		\begin{equation*}
				\EE \left[e^{ - \int f(x,j) \Npa_n(\ud x, \ud j) } \right] \to \EE \left[\exp\left\{ - W \int \left(1-e^{-f(x,l)}\right) e^{-x}\ud x \rho(\ud l) \right\}\right]
		\end{equation*}
		for any continuous, nonnegative function $f\colon (-\infty, \infty]\times \NN$ such that $f(x,i) =0$ for sufficiently large $|x|$.
		In particular, $\Npa_n$ converges weakly in
		$\M ((-\infty, \infty]\times \NN)$ to $\Npa$.
	\end{prop}
	
	\begin{proof}
		For a function $f$ as in the statement we have
		\begin{equation*}
			\EE \left[e^{ - \int f(x,j) \Npa_n(\ud x, \ud j) } \right]
			= \EE \left[ \exp\left\{ - \sum_{|v|\leq n} f\left(\frac{X_v -b_n}{a_n}, n-|v|\right)  \right\} \right].
		\end{equation*}
		Conditioning on the Galton-Watson process we get
		\begin{equation*}
			\EE\left[ \prod_{k=1}^n \EE\left[\exp\left\{ -  f\left(\frac{X -b_n}{a_n}, n-k\right)  \right\}\right]^{Z_k} \right].
		\end{equation*}
		Using the facts that
		\begin{equation*}
			m^n \PP[(X-b_n)/a_n \in \ud x] \to^\nu e^{-x}\ud x
		\end{equation*}
		and that for any fixed $l$
		\begin{equation*}
			m^{-n}Z_{n-l} \to m^lW,
		\end{equation*}
		we infer that for any fixed $l$,
		\begin{equation*}
			\EE\left[\exp\left\{ -  f\left(\frac{X -b_n}{a_n}, l\right)  \right\}\right]^{Z_{n-l}} \to \exp\left\{ - m^{-l}W \int \left(1-e^{-f(x,l)}\right) e^{-x}\ud x \right\},
		\end{equation*}
since $f(\cdot, l) \in C_c^+((-\infty, \infty])$.
		Combine this with the proviso concerning the support of $f$ and use a standard approximation of infinite product by finite ones, to get
		\begin{multline*}
			\EE \left[e^{ - \int f(x,j) \Npa_n(\ud x, \ud j) } \right] = \EE\left[ \prod_{k=1}^n \EE\left[\exp\left\{ -  f\left(\frac{X -b_n}{a_n}, n-k\right)  \right\}\right]^{Z_k} \right]\\ \to \exp\left\{ - \sum_{l=0}^\infty m^{-l}W \int \left(1-e^{-f(x,l)}\right) e^{-x}\ud x \right\}
		\end{multline*}
as $n$ tends to infinity.
	\end{proof}

	We will use the last proposition to show %in favour of
a limit theorem for
	\begin{equation*}
      		\Npb_n = \sum_{x \in \mathcal{T}_n} \delta_{(X_x-b_n)/a_n}\otimes \delta_{E_{n-|x|}(x)}.
	\end{equation*}
	The proof will rely on the observation that the random variables $\{E_{n-|x|}(x) \}_{x \in \mathcal{T}_n}$
	are independent and are independent of $\{X_{x}\}_{x \in \mathcal{T}_n}$ given $\mathcal{T}_n$. Furthermore,
	\begin{equation*}
		\PP[ E_{n-|x|}(x) \in \cdot | x \in \mathcal{T}_n] = \PP[Z_{n-k} \in \cdot | Z_{n-k}>0] \quad \mbox{for $|x|=k$}.
	\end{equation*}
	If we thus consider iid copies $\{Z^{(v)}\}_{v \in \mathbb{U}}$ of the underlying Galton-Watson process $Z = \{ Z_{k}\}_{k \in \NN}$ we
	can conclude that
	\begin{equation*}
      		\Npb_n = \sum_{x \in \mathcal{T}_n} \delta_{(X_x-b_n)/a_n}\otimes \delta_{E_{n-|x|}(x)} \stackrel{d}{=}\sum_{x \in \mathcal{T}_n} \delta_{(X_x-b_n)/a_n}\otimes \delta_{Z^{(x)}_{n-|x|}}.
	\end{equation*}
	We will use this representation in the proof of the next proposition. Define
	\begin{equation*}
		\rho^*(\ud k) = \sum_{j=0}^\infty m^{-j} \PP[ Z_j \in \ud k]
	\end{equation*}
	and a Poisson point process $\Npb$ on $\RR \times \NN$ given via
	\begin{equation*}
		\Npb = \sum_{k}\delta_{\ell_k - \log(aW)}\otimes \delta_{\iota_k^*},
	\end{equation*}
	where $\{ (\ell_k, \iota_k^*)\}_{k}$ is a Poisson point process
	with intensity $e^{-x}\ud x \rho^*(\ud j)$. Then $\Npb$ has a Laplace functional of the form
	\begin{equation*}
		\EE \left[e^{ - \int f(x,j) \Npb(\ud x, \ud j) } \right] = \EE \left[\exp\left\{ - W \int \left(1-e^{-f(x,l)}\right) e^{-x}\ud x \rho^*(\ud l) \right\}\right].
	\end{equation*}

	\begin{prop}\label{prop:4:npb}
		Under the Assumption~\ref{as:galton} and Assumption~\ref{as:step1}, we have
		\begin{equation*}
				\EE \left[e^{ - \int f(x,j) \Npb_n(\ud x, \ud j) } \right] \to \EE \left[e^{ - \int f(x,j) \Npb(\ud x, \ud j) } \right]
		\end{equation*}
		for any continuous, nonnegative function $f\colon (-\infty, \infty]\times \NN$ such that $f(x,i) =0$ for sufficiently large $|x|$.
		In particular, $\Npb_n$ converges weakly in
		$\M ((-\infty, \infty]\times \NN)$ to $\Npb$.
	\end{prop}
	\begin{proof}
		We have
		\begin{equation*}
			\int f(x,j) \Npb_n(\ud x, \ud j) \stackrel{d}{=} \sum_{x \in \mathcal{T}_n}  f \left((X_x-b_n)/a_n,Z^{(x)}_{n-|x|} \right) .
		\end{equation*}
		On the event $\{\mathcal{A}_n^{(2)}=\emptyset\}$, for sufficiently large $n$,
		\begin{equation*}
			\sum_{x \in \mathcal{T}_n}  f \left((X_x-b_n)/a_n,Z^{(x)}_{n-|x|} \right) =\sum_{|x| \leq n}  f \left((X_x-b_n)/a_n,Z^{(x)}_{n-|x|} \right) .
		\end{equation*}
		With the last representation at hand, we can consider
		\begin{equation*}
			f^*(x,j) = - \log \EE[\exp \{ - f(x, Z_{j})\} ]
		\end{equation*}
		and write, by conditioning on $X_x$'s,
		\begin{equation*}
			\EE\left[ \exp \left\{ - \int f(x,j) \Npb_n(\ud x, \ud j) \right\} \right]
			= \EE\left[ \exp \left\{ - \int f^*(x,j) \Npa_n(\ud x, \ud j) \right\} \right].
		\end{equation*}
		Since the function $f^*$ satisfies the hypothesis of Proposition~\ref{prop:4:banana} we conclude the proof.
	\end{proof}

	We are now ready to prove the main proposition.

	\begin{prop}\label{prop:4:tn}
		Under the Assumption~\ref{as:galton} and Assumption~\ref{as:step1},
		\begin{equation*}
			\scale_{a^{-1}_n}\shift_{-b_n} \Tp_n \Rightarrow  \Pp,
		\end{equation*}
		where $\Pp$ is a point process given in~\eqref{eq:2:defV}.
	\end{prop}
	\begin{proof}
    		Fix $f\in C_c^+(-\infty, \infty]$. As noted previously for sufficiently large $n$, on the event $\{\mathcal{A}_n^{(2)}=\emptyset\}$,
		\begin{equation*}
			\sum_{|x|=n} f\left(\frac{T(x)-b_n}{a_n}\right) =\sum_{v \in \mathcal{T}_n} f\left(\frac{X_v -b_n}{a_n}\right) E_{n-|v|}(v).
		\end{equation*}
		The last sum is equal to
		\begin{equation*}
			\int f(x) j  \: \Npb_n(\ud x, \ud j).
		\end{equation*}
		By an appeal to Proposition~\ref{prop:4:npb},
		\begin{equation*}
			\int f(x) j  \: \Npb_n(\ud x, \ud j) \to^d \int f(x) j  \: \Npb(\ud x, \ud j).
		\end{equation*}
		The result follows, since by construction
 		\begin{equation*}
			\int f(x) j  \: \Npb(\ud x, \ud j)\stackrel{d}{=} \int f(x)   \: \Pp(\ud x).
		\end{equation*}
%On the event, where there are at most one macroscopic jump along each ancestral line,
%		we can rewrite the integral of $f$ with respect to $\scale_{a^{-1}_n}\shift_{-b_n}\tilde \Pp_n$ as
%		\begin{equation*}
%      \sum_{|v|\leq n} f((X_v-b_n)/a_n) \#\{ |x|=n\: :\: x \geq v\}.
%		\end{equation*}
%    We can view the above as an integral of $f$ with respect to an independent marking of a binomial point process. More precisely
%
%
%
%    Then, as $n$ tends to infinity, $\widehat{\Pp}_n$ converge weakly to a Poisson point process with intensity
%    \begin{equation*}
%      e^{-x} m^{-j}\ud x \ud j
%    \end{equation*}
%    For $|v|\leq n$ let $\{Z^{(v)}_k\}_{k \in \NN}$ be an independent copy of the underlying Galton-Watson process. Consider an independent marking of
%    $\widehat{\Pp}_n$ given via
%    \begin{equation*}
%      \widehat{\Pp}_n^*=\sum_{|x|\leq n} \delta_{(X_x-b_n)/a_n}\otimes \delta_{Z_{n-|v|}^{(v)}}
%    \end{equation*}
%    which converges to a Poisson point process on $(-\infty, \infty]\times \NN$ with intensity
%    \begin{equation*}
%      e^{-x}\ud x \PP[T_1 \in \ud k]
%    \end{equation*}
	\end{proof}

	We are now in position to give a final touch to the proof of our first main result.

	\begin{proof}[Proof of Theorem~\ref{thm:2:1}]
		At this point it suffices to approximate $\Pp_n$ via $\Tp_n$. Fix $\varepsilon>0$.
		Any function $f$ from $C_c(-\infty, \infty]$ is uniformly continuous. One can thus
		has
		\begin{equation*}
			\omega_f(\delta) = \sup \{ |f(x)- f(y)| \: : \: |x-y| <\delta \} \to 0
		\end{equation*}
		as $\delta\to 0^+$.
		 Note that
		\begin{multline*}
			\left|  \int f(x) (\Pp_n-\Tp_n)(\ud x)\right| \leq  \\ \int f(x) \Pp_n^{(\mathcal{A})}(\ud x)
		+ \omega_f\left( M_n^{(\mathcal{A})}\right) \Tp_n\left[{\rm supp}(f)+\left(-M_n^{(\mathcal{A})}, M_n^{(\mathcal{A})}\right)\right].
		\end{multline*}
		By an appeal to Lemma~\ref{lem:4:ao} the first term vanishes. The second term vanishes as well by
		Proposition~\ref{prop:4:tn} and Lemma~\ref{lem:4:tea}.
	\end{proof}

\subsection{The sublogarithmic case}
	The arguments go along the lines of the proof of Theorem~\ref{thm:2:1}. We will therefore skip some of the more simpler arguments.
	Firstly recall the generalised inverse of $L$ given via
	\begin{equation*}
		L^{-1}(y) = \inf \{ s \in \RR \: : \: L(s) \geq y\}.
	\end{equation*}
	Since $L$ is right continuous, $L(x)\geq y$ if and only if $x \geq L^{-1}(y)$. Consider
	\begin{equation*}
		\mathcal{B}_n = \{ x \in \mathbb{T}_n \: : \: L(T(x)) \leq n(\log m) /2\}.
	\end{equation*}
	We will first show that the positions of the particles from $\mathcal{B}_n$ do not contribute to the limit of $\Pp_n$.

	\begin{lem}
		Under the Assumption~\ref{as:galton} and Assumption~\ref{as:step2},
		\begin{equation*}
			\sum_{x \in \mathcal{B}_n} \delta_{L(V(x)) - n\log m } \Rightarrow o
		\end{equation*}	
		in $\M(\bR)$. % where $o$ denotes the null measure.
	\end{lem}

	\begin{proof}
		As in the proof of Lemma~\ref{lem:4:ao} it is sufficient to show that for any $C \in \RR$,
		\begin{equation*}
			\PP[L(S_n) > n \log m +C, \:  L(N_n) \leq  3n(\log m) /4] = o(m^{-n}).
		\end{equation*}
		We will once again apply Lemma~\ref{lem:4:ao}. Denote $\ell_n = L^{-1}(3n(\log m) /4)$, $\varpi_n =L^{-1}(n \log m  +C)$. Note that
		$\ell_n = o(\varpi_n)$. Consider $s = n\log m / (2\ell_n)$ and write
		\begin{multline}\label{eq:4:cookie}
			\PP[L(S_n) > n\log m  +C, \:  L(N_n) \leq  3n(\log m) /4] \leq \\ e^{-  s \varpi_n } \EE\left[e^{sX} \ind_{\{ X \leq \ell_n \}}\right]^n.
		\end{multline}
		We will present a suitable bound for the integral which appears on the right hand side.
		First note that
		\begin{equation*}
			\EE\left[e^{sX} \ind_{\{ X \leq \ell_n/n \}}\right] \leq m.
		\end{equation*}
		For any $\varepsilon>0$ consider
		\begin{multline*}
			\EE\left[e^{sX} \ind_{\{ \ell_n/n \leq  x \leq \varepsilon \ell_n \}}\right]
			\leq \int_{\ell_n/n}^{\varepsilon \ell_n} s e^{sy} \PP[X>y] \ud y + e \PP[X>\ell_n/n] \\
			\leq \int_{1/n}^{\varepsilon}  n \log m e^{s\ell_nt - L(\ell_n t)} \ud t +O(1).
		\end{multline*}
		Using Potter bounds,
		\begin{equation*}
			L(\ell_n t) \geq t^{\delta}L(\ell_n) \geq \varepsilon^{\delta-1} t L(\ell_n)
		\end{equation*}
		for $t \in [1/n, \varepsilon]$. If we plug this into the integral we will get, with sufficiently small $\varepsilon$,
		\begin{equation*}
			\EE\left[e^{sX} \ind_{\{ \ell_n/n \leq  x \leq \varepsilon \ell_n \}}\right]  = o(1).
		\end{equation*}
		Finally, for the last part write
		\begin{multline*}
			\EE\left[e^{sX} \ind_{\{ \varepsilon \ell_n \leq  x \leq \ell_n \}}\right]  \leq \int_{\varepsilon \ell_n}^{\ell_n} s e^{sy} \PP[X>y] \ud y + e^{\varepsilon s \ell_n } \PP[X>\varepsilon \ell_n]  \\
				 =\int_{\varepsilon}^1  n \log m  e^{s\ell_nt - L(\ell_n t)} \ud t +o(1)
		\end{multline*}
		and the last integral converges to $0$ exponentially fast by the merit of uniform slow variation of $L$. This constitutes
		\begin{equation*}
			\EE\left[e^{sX} \ind_{\{ x \leq \ell_n \}}\right] \leq C.
		\end{equation*}
		If we now go back to~\eqref{eq:4:cookie}, we get
		\begin{multline*}
			\PP[L(S_n) > n \log m +C, \:  L(N_n) \leq  n(\log m) /2] \\
				\leq \exp\{ -n \log m \varpi_n/(2\ell_n) + Cn\}.
		\end{multline*}
		Since $\varpi_n/\ell_n \to \infty$ our claim follows.
	\end{proof}

	We will now state the result that allows to approximate $\Pp_n$ via $\Tp_n$. Define
	\begin{equation*}
		M_n^{(B)} = \max \{ L(V(x))-L(T(x)) \: : \: |x|=n , \: L(T(x)) > 3n(\log m )/4\}.
	\end{equation*}

\begin{lem}
	Let the Assumption~\ref{as:galton} and Assumption~\ref{as:step2} be satisfied. Then, as $n \to \infty$,
	\begin{equation*}
		M_n^{(B)} %\to^\PP
\convprob 0.
	\end{equation*}
\end{lem}
\begin{proof}
%	First
%	By Karamata representation
%	\begin{equation*}
%		L(t) = c(t) \exp \left\{ \int_1^x \varepsilon(t)/t \ud t \right\}
%	\end{equation*}
%	for $c(t) \to c >0$ and $\varepsilon \to 0$. Write $L_0(t) = L(t) /c(t)$. We have
%	\begin{align*}
%		L_0(x+y)-L_0(x) & = L_0(x) \left(\exp \left\{ \int_x^{x+y} \varepsilon(t)/t \ud t \right\} -1 \right) \\
%			& \leq L_0(x) \left( \left(1+\frac yx \right)^{\varepsilon(x)} -1 \right) \leq L_0(x) \varepsilon(x) y/x
%	\end{align*}

	We recall that $\{ S_n \}_{n \in \NN}$ is a random walk generated by $X$ and %. That is $S_0=0$ and for $n \geq 1$, $S_n = X_1+X_2+\ldots +X_n$, where
%	$\{X_k\}_{k \in \NN}$ are iid copies of $X$. We will denote also
$N_n = \max_{k \leq n}X_k$.
	Using Potter bounds for any $\delta>0$, $x\in \RR$ and $y>0$,
	\begin{align*}
		L(x+y)-L(x) & \leq C L(x) \left( \left(1+\frac yx \right)^{\delta} -1 \right).
	\end{align*}
	The right-hand side, by the Bernoulli inequality, is further bounded via $C L(x) \delta  y/x$.
	Thus for $S_n^o = S_n-N_n$ we have, for sufficiently large $n$,
	\begin{equation*}
		\{ L(S_n) - L(N_n) >\varepsilon, \: N_n >  L^{-1}(3n \log m /4) \} \subseteq \left\{ S_n^o > \varepsilon s_1/(\delta \log s_1) \right\},
	\end{equation*}
	where $s_1= L^{-1}(3n(\log m) /4) $.
	Now we will show that if $L(x) = o(\log x)$ then
	\begin{equation}\label{eq:4:claim}
		\lim_{n \to \infty}\frac{s_1/\log s_1 }{s_2 } =\infty \qquad \text{where $s_2 = L^{-1}(2n(\log m) /3)$.}
	\end{equation}
	By Karamata representation
	\begin{equation*}
		L(t) = c(t) \exp \left\{ \int_1^t \varepsilon(x)/x \ud x \right\}
	\end{equation*}
	for $c(t) \to c >0$ and $\varepsilon(x) \to 0$.
	Since $L(x) = o(\log x)$, we must have
	\begin{equation*}
		\varepsilon(z) \leq 1/\log z
	\end{equation*}
	for sufficiently large $z$.
	Then
	\begin{align*}
		3n (\log m) /4 \leq L(s_1) & = L(s_2/2) c(s_1)/c(s_2/2) \exp \left\{ \int_{s_2/2}^{s_1} \varepsilon(z)/z \ud z \right\}\\
			& \leq L(s_2/2) c(s_1)/c(s_2/2) \frac{\log s_1}{\log(s_2/2)}.
	\end{align*}
	This leads to
	\begin{equation*}
		(1+o(1))9/8 \leq  \frac{\log s_1}{\log s_2}
	\end{equation*}
	which means that for sufficiently large $n$, by the monotonicity of the function $x \mapsto x/\log x$,
	%\begin{equation*}
	%	\frac{s_2^{1/9}}{8 (\log s_2)/9} \leq \frac{s_1/\log s_1}{s_2}	.
	%\end{equation*}
\begin{equation*}
		\frac{s_2^{1/8}}{9 (\log s_2)/8} \leq \frac{s_1/\log s_1}{s_2}	.
	\end{equation*}
	Since $s_2$ grows faster than exponentially in $n$ this secures~\eqref{eq:4:claim}.
	We infer that
	\begin{multline*}
		\PP \left[L(S_n) - L(N_n) >\varepsilon, \: N_n >  L^{-1}(3n (\log m) /4) \right] \leq \\
			\PP[ N_n >  L^{-1}(3n (\log m)/4) , S_n^o > n L^{-1}(2n (\log m )/3)]\leq C n^2 m^{-n(2/3+3/4)}.
	\end{multline*}
	The lemma now follows by a first moment argument since for any $\varepsilon>0$,
	\begin{align*}
		\PP[M_n^{(B)} > \varepsilon] & \leq m^n \PP \left[L(S_n) - L(N_n) >\varepsilon, \: N_n >  L^{-1}(3n(\log m)/4) \right]\\
			&  \leq n^2 m^{-5n/12}.
	\end{align*}
\end{proof}

\begin{proof}[Proof of Theorem~\ref{thm:L:sublog}]
	We apply the same arguments as in the proof of our first result. This time we consider the stopping line
	\begin{multline}\label{eq:4:line2}
		\mathcal{S}_n =
		\Big\{ |v|\leq n \: : \: \exists x \in \mathbb{T}_n\setminus \mathcal{B}_n, \: x \geq v, \\
			 \max_{y<v}L(X_y) \leq  3n (\log m) /4, \: L(X_v)> 3n (\log m) /4 \Big\}.
	\end{multline}
	Now for any compactly supported and continuous $f$, with high probability
	\begin{equation*}
		\sum_{|x|=n} f\left(L(T(x)) - n \log m \right) =\sum_{v \in \mathcal{T}_n} f\left(L(X_v) - n \log m \right) E_{n-|v|}(v).
	\end{equation*}
	One then uses the same procedure taking into account that
	\begin{equation*}
		m^n \PP[L(X)-n \log m \in \ud x] \to^\nu e^{-x}\ud x.
	\end{equation*}
\end{proof}

\bibliographystyle{abbrv}
\bibliography{BRWwithSVT}

\end{document}